\documentclass[12pt,reqno]{amsart}
\usepackage{times}
\usepackage{amssymb}
\usepackage[dvips]{graphicx}

\newtheorem{thm}{Theorem}
\newtheorem{lem}[thm]{Lemma}
\newtheorem{cor}[thm]{Corollary}

\newtheorem{rem}[thm]{Remark}

\begin{document}

\title {\bf Gaussian Integral Means of Entire Functions}
\thanks{This work was in part supported by NSERC of Canada and completed during the first-named author's visit (2012.9-12) to Memorial University.}

\author{Chunjie Wang}
\address{Chunjie Wang, Department of Mathematics, Hebei University of Technology,
Tianjin 300401, China} \email{wcj@hebut.edu.cn}

\author{Jie Xiao}
\address{Jie Xiao, Department of Mathematics and Statistics, Memorial University
of Newfoundland, St. John's, NL A1C 5S7, Canada}
\email{jxiao@mun.ca}

\begin{abstract} For an entire function $f:\mathbb C\mapsto\mathbb C$ and a triple $(p,\alpha, r)\in (0,\infty)\times(-\infty,\infty)\times(0,\infty]$, the Gaussian integral means of $f$ (with respect to the area measure $dA$) is defined by
$$
{\mathsf M}_{p,\alpha}(f,r)=\Big(
{\int_{|z|<r}e^{-\alpha|z|^2}dA(z)}\Big)^{-1}{\int_{|z|<r}|f(z)|^p{e^{-\alpha|z|^2}}dA(z)}.
$$
Via deriving a maximum principle for ${\mathsf M}_{p,\alpha}(f,r)$, we establish not only Fock-Sobolev trace inequalities associated with ${\mathsf M}_{p,p/2}(z^m f(z),\infty)$ (as $m=0,1,2,...$), but also convexities of $r\mapsto\ln {\mathsf M}_{p,\alpha}(z^m,r)$ and $r\mapsto {\mathsf M}_{2,\alpha<0}(f,r)$ in $\ln r$ with $0<r<\infty$.
\end{abstract}


\maketitle

\section{Introduction}\label{s1}

Let $dA$ be the Euclidean area measure on the finite complex plane $\mathbb{C}$.
Suppose $\alpha$ is real and $0<p<\infty$. For any entire function $f:\mathbb C\mapsto \mathbb C$,
we consider its Gaussian integral means
$$
{\mathsf M}_{p,\alpha}(f,r)=\frac{\int_{|z|<r}|f(z)|^p{e^{-\alpha|z|^2}}dA(z)}
{\int_{|z|<r}e^{-\alpha|z|^2}dA(z)}\quad\forall\quad r\in (0,\infty).
$$
Upon writing
$$
\begin{cases}
M(r)=\int_0^{2\pi}|f(re^{i\theta})|^pd\theta;\\
v(r)=r e^{-\alpha r^2};\\
i=\sqrt{-1} - \hbox{the\ imaginary\ unit},
\end{cases}
$$ 
we get
$$
\frac{d}{dr}{\mathsf M}_{p,\alpha}(f,r)=\frac{v(r)\int_0^r\big(M(r)-M(s)\big)v(s)ds}{2\pi\left(\int_0^r v(s)ds\right)^2}\ge 0,
$$
and hence the function $r\mapsto \mathsf{M}_{p,\alpha}(f,r)$ is strictly increasing on $(0,\infty)$ unless $f$ is constant. Consequently, 
letting $r\to 0$ and $r\to\infty$ in ${\mathsf M}_{p,\alpha}(f,r)$ respectively, we find the following maximum principle for $r\in (0,\infty)$: 
\begin{align*}
|f(0)|^p&={\mathsf M}_{p,\alpha}(f,0)\le {\mathsf M}_{p,\alpha}(f,r)\\
&\le {\mathsf M}_{p,\alpha}(f,\infty)=\frac{\int_{\mathbb C}|f(z)|^pe^{-\alpha|z|^2}dA(z)}
{\int_{\mathbb C}e^{-\alpha|z|^2}dA(z)}
\end{align*}
with equality if and only if $f$ is a constant. 

Besides the above maximum principle we are here motived mainly by \cite{XZ, L, L1, X1, WZ, XX, CZ} to take a further look at the Gaussian integral means $\mathsf{M}_{p,\alpha}(f,r)$ from two perspectives. The first is to treat the last inequality as a space embedding: if $d\mu_r(z)=1_{|z|<r}\,dA(z)$ (with $1_E$ being the characteristic function of $E\subset\mathbb C$) then 
$$
\int_{\mathbb C}|f(z)e^{-\frac{|z|^2}{2}}|^p\,d\mu_r(z)\le\left({\int_{|z|<r}e^{-\frac{p|z|^2}{2}}\,dA(z)}\right)\mathsf{M}_{p,p/2}(f,\infty).
$$
Such an interpretation leads to characterizing a given nonnegative Borel measure $\mu$ on $\mathbb C$ such that the following Fock-Sobolev trace inequality
\begin{align*}
&\|f\|_{L^q(\mathbb C,\mu)}\\
&\equiv\left(\int_{\mathbb C}|f(z)e^{-\frac{|z|^2}{2}}|^q\,d\mu(z)\right)^\frac1q\\
&\lesssim\Big(\mathsf{M}_{p,p/2}(z^mf(z),\infty)\Big)^\frac1p\\
&\approx\left(\int_{\mathbb C}|z^mf(z)e^{-\frac{|z|^2}{2}}|^p\,dA(z)\right)^\frac1p\\
&\equiv\|f\|_{\mathcal F^{p,m}}
\end{align*}
holds for all holomorphic functions $f:\mathbb C\mapsto\mathbb C$ in $\mathcal F^{p,m}$. In the above and below:
\medskip

\noindent $\bullet$\quad $0<p,q<\infty$;\\ 
$\bullet$\quad $\mathsf X\lesssim \mathsf Y$ (i.e. $\mathsf Y\gtrsim \mathsf X$)  means that there is a constant $c>0$ such that $\mathsf X\le c\mathsf Y$ - moreover - $\mathsf{X}\approx\mathsf{Y}$ is equivalent to $\mathsf X\lesssim \mathsf Y\lesssim\mathsf{X}$;\\
$\bullet$\quad $m$\quad is nonnegative integer;\\  
$\bullet$\quad $\mathcal F^p=\mathcal F^{p,0}$ and $\mathcal F^{p,m}$ stand for the so-called Fock space and Fock-Sobolev space of order $m\ge 1$ respectively. Interestingly, for an entire function $f:\mathbb C\mapsto\mathbb C$ one has (cf. \cite{CZ}):
$$
f\in \mathcal F^{p,m}\Longleftrightarrow |f(0)|+\cdots+|f^{(m-1)}(0)|+\|f^{(m)}\|_{\mathcal F^{p,0}}<\infty.
$$
$\bullet$\quad $B(a,r)=\{z\in\mathbb C:\ |z-a|<r\}$ is the Euclidean disk centered at $a\in\mathbb C$ with radius $r>0$.
\medskip

\noindent As stated in Theorem \ref{t2a} of Section \ref{s2}, the above-required measure is fully determined by
$$
\begin{cases}
\sup_{a\in\mathbb C}\frac{\mu(B(a,r))}{(1+|a|)^{qm}}<\infty\quad\hbox{as}\quad 0<p\le q<\infty;\\
\int_{\mathbb C}\Big(\frac{\mu(B(a,r))}{(1+|a|)^{qm}}\Big)^\frac{p}{p-q}\,dA(a)<\infty\quad\hbox{as}\quad 0<q<p<\infty.
\end{cases}
$$
As a particularly interesting and natural by-product of this characterization, we can also use the Taylor expansion of an entire function at the origin to get the optimal Gaussian Poincar\'e inequality (see \cite[(1.6)]{Pe} as well as \cite[p. 115]{Lex} and \cite[Theorem 1]{Yau} for the endpoint case corresponding to $f\in \mathcal F^{1,1}$ with $f(0)=0$) 
$$
\int_{\mathbb C}|f(z)e^{-\frac{|z|^2}{2}}|^2\,dA(z)-\pi|f(0)|^2\le
\int_{\mathbb C}|f'(z)e^{-\frac{|z|^2}{2}}|^2\,dA(z)\ \ \forall\ \ f\in\mathcal F^{2,1}
$$
which, plus the foregoing maximum-principle-based estimate (cf. \cite[(1)]{CZ})
$$
|f'(z)|e^{-\frac{|z|^2}{2}}\le(2\pi)^{-1}\int_{\mathbb C}|f'(z)e^{-\frac{|z|^2}{2}}|\,dA(z)\ \ \forall\ \ f\in\mathcal F^{1,1},
$$
derives the following Gaussian isoperimetric-Sobolev inequality $f\in\mathcal F^{1,1}$:
$$
\int_{\mathbb C}|f(z)e^{-\frac{|z|^2}{2}}|^2\,dA(z)-\pi|f(0)|^2\le(2\pi)^{-1}\left(\int_{\mathbb C}|f'(z)e^{-\frac{|z|^2}{2}}|\,dA(z)\right)^2
$$
whose sharp form is 
$$
\int_{\mathbb C}|f(z)e^{-\frac{|z|^2}{2}}|^2\,dA(z)-\pi|f(0)|^2\le(4\pi)^{-1}\left(\int_{\mathbb C}|f'(z)e^{-\frac{|z|^2}{2}}|\,dA(z)\right)^2
$$
since this inequality can be proved valid for the entire functions $f(z)=z^k$ with $k=1,2,3,...$ through a direct computation with the polar coordinate system, the mathematical induction and the inequality for the gamma function $\Gamma(\cdot)$ below:
$$
\frac{\Gamma\big(\frac{k+1}{2}\big)}{\Gamma\big(\frac{k}{2}\big)}\le\sqrt{\frac{k+1}{2}}.
$$

The second is to decide: when $\ln r\mapsto \ln \mathsf{M}_{p,\alpha}(z^k,r)$ is convex for $r\in (0,\infty)$, namely, when the Gaussian Hadamard Three Circle Theorem below
$$
\Big(\ln\frac{r_2}{r_1}\Big)\ln\mathsf{M}_{p,\alpha}(z^k,r)\le\Big(\ln\frac{r_2}{r}\Big)\ln\mathsf{M}_{p,\alpha}(z^k,r_1)+\Big(\ln\frac{r}{r_1}\Big)\ln\mathsf{M}_{p,\alpha}(z^k,r_2)
$$
holds for $0< r_1\le r\le r_2<\infty$. The expected result is presented in Theorem \ref{t3a} of Section \ref{s3}, saying that for a nonnegative integer $k$ and a positive number $p$,
$$
\begin{cases}
\ln r\mapsto\ln\mathsf{M}_{p,\alpha}(z^k,r)\ \hbox{is\ concave\ as}\ r\in(0,\infty)\ \hbox{under}\ 0<\alpha<\infty;\\
\exists\ c\in (0,\infty)\ \ni\ \ln r\mapsto \ln\mathsf{M}_{p,\alpha}(z^k,r)\ \hbox{is\  convex\ and\ concave}\\
 \hbox { as}\ r\in(0,c]\ \hbox{and}\  r\in[c,\infty)\  \hbox{respectively}\ \hbox{under}\ -\infty<\alpha\le 0.
\end{cases}
$$
As a consequence, we have that if $-\infty<\alpha, -p<0$ then the function $\ln r\mapsto \ln \mathsf{M}_{p,\alpha}(z^k,r)$ is convex as $r\in (0,\sqrt{(2+pk)/(-2\alpha)}]$ and hence the function $\ln r\mapsto \ln \mathsf{M}_{2,\alpha}(f,r)$ is convex as $r\in(0,\sqrt{1/(-\alpha)}]$ for any entire function $f:\mathbb C\mapsto\mathbb C$. In other words,
$$
\Big(\ln\frac{r_2}{r_1}\Big)\ln\mathsf{M}_{2,\alpha}(f,r)\le\Big(\ln\frac{r_2}{r}\Big)\ln\mathsf{M}_{2,\alpha}(f,r_1)+\Big(\ln\frac{r}{r_1}\Big)\ln\mathsf{M}_{2,\alpha}(f,r_2)
$$
when $0< r_1\le r\le r_2<\sqrt{1/(-\alpha)}$. However, as proved in Remark \ref{r3a} via considering the entire function $1+z$, the last convexity cannot be extended to $(0,\infty)$.
 
\section{Trace inequalities for Fock-Sobolev spaces}\label{s2} 

We need two lemmas. The first lemma comes from \cite{CZ} and \cite{Z, Z1, JPR, W}.

\begin{lem}\label{l2a} Let $p,\sigma,a,t,\lambda\in (0,\infty)$.

\item{\rm(i)} If $m$ is a nonnegative integer, $p_m(z)$ is the Taylor polynomial
of $e^z$ of order $m-1$ (with the convention that $p_0=0$), and $b>-(mp+ 2)$, then 
$$\int_{\mathbb{C}}|e^{z\overline{w}}-p_m(z\overline{w})|^pe^{-a|w|^2}|w|^b dA(w)\lesssim |z|^be^{\frac{p^2}{4a}|z|^2}\quad\forall\quad |z|\ge\sigma.
$$
Furthermore, this last inequality holds also for all $z\in\mathbb C$ when $b \le pm$.

\item{\rm(ii)} If $f:\mathbb C\mapsto\mathbb C$ is an entire function, then
$$
\left|f(z)e^{-\frac{\lambda}{2}|z|^2}\right|^p \lesssim\int_{B(z,t)}\left|f(w)e^{-\frac{\lambda}{2}|w|^2}\right|^p dA(w)\quad\forall\quad z \in \mathbb{C}.
$$

\item{\rm(iii)} There exists a positive constant $r_0$ such that for any
$0<r<r_0$, the Fock space $\mathcal F^{p}$ exactly consists of all functions
$f=\sum_{w\in r\mathbb{Z}^2}c_wk_w$, where 
$$
\begin{cases}
k_w(z)=\exp({z\bar{w}-|w|^2/2});\\
\{c_w: w\in r\mathbb{Z}^2\}\in l^p;\\
\|\{c_w\}\|_{l^p}=\big(\sum_{w\in r\mathbb Z^2}|c_w|^p\big)^\frac1p;\\
\mathbb Z^2=\{n+im:\ n,m=0,\pm 1,\pm 2,...\};\\
r\mathbb Z^2=\{r(n+im):\ n,m=0,\pm 1,\pm 2,...\}.
\end{cases}
$$
Moreover
$$
\|f\|_{\mathcal F^p}\approx \inf
\|\{c_w\}\|_{l^p}\quad\forall\quad f \in \mathcal F^p,
$$
where the infimum is taken over all sequences $\{c_w\}$ giving rise to
the above decomposition.
\end{lem}

The second lemma is the so-called Khinchine's inequality, which can be found, for example, in \cite{L1}.
\begin{lem}\label{l2b}
Suppose $p\in (0,\infty)$ and $c_j\in\mathbb C$. For the integer part $[t]$ of $t\in (0,\infty)$  let
 $$
 r_0(t)=\left\{\begin{array}{ll}
 1, & 0\le t-[t]<1/2 \\
-1, & 1/2\le t-[t]<1
\end{array}
\right.
$$
and 
$$
r_j(t)=r_0(2^jt)\quad\forall\quad j=1,2,\cdots.
$$ 
Then
$$
\left(\sum_{j=1}^m|c_j|^2\right)^\frac{1}{2}\thickapprox\left(\int_0^1\left|\sum_{j=1}^mc_jr_j(t)\right|^pdt\right)^\frac{1}{p}.
$$
\end{lem}

As the main result of this section, the forthcoming family of analytic-geometric trace inequalities for the Fock-Sobolev spaces is a natural generalization of the so-called diagonal Carleson measures for the Fock-Sobolev spaces in \cite{CZ}.  

\begin{thm}\label{t2a} Let $m$ be a nonnegative integer, $r\in (0,\infty)$, and $\mu$ be a nonnegative Borel measure on $\mathbb{C}$. 

\item{\rm(i)} If $0<p\le q<\infty$, then
$$
\|f\|_{L^q(\mathbb C,\mu)}\lesssim\|f\|_{\mathcal F^{p,m}}\quad\forall\quad f\in \mathcal F^{p,m}
$$
when and only when 
$$
\sup_{a \in\mathbb{C}}\frac{\mu(B(a,r))}{(1+|a|)^{mq}}<\infty.
$$
Equivalently, $a\mapsto\mu(B(a,r))(1+|a|)^{-qm}$ is of class $L^{\infty}(\mathbb C)$.

\item{\rm(ii)} If $0<q<p<\infty$, then
$$
\|f\|_{L^q(\mathbb C,\mu)}\lesssim\|f\|_{p,m}\quad\forall\quad
f\in \mathcal F^{p,m}
$$
when and only when
$$
\sum_{a\in s\mathbb Z^2}\left(\frac{\mu(B(a,r))}{(1+|a|)^{mq}}\right)^\frac{p}{p-q}<\infty\quad\hbox{where}\quad s\in (0,\infty).
$$
Equivalently, $a\mapsto\mu(B(a,r))(1+|a|)^{-qm}$ is of class $L^{p/(p-q)}(\mathbb C)$.
\end{thm}

\begin{proof} (i) Suppose $0<p\le q<\infty$. The following argument is similar to that of Theorem 10 in \cite{CZ}.

 Assume firstly that $\|f\|_{L^q(\mathbb C,\mu)}\lesssim\|f\|_{\mathcal F^{p,m}}$ holds for all
$f\in \mathcal F^{p,m}$. Taking $f=1$ shows that $\mu(K)\lesssim 1$ for any compact set $K\subset\mathbb C$.

Fix any $a\in\mathbb C$ and let
$$
f(z) = (e^{z\overline{a}}-p_m(z\overline{a}))/z^m
$$
in the last assumption. Then Lemma \ref{l2a} (i) implies 
$$\int_{\mathbb{C}}\left|\frac{e^{z\overline{a}}-p_m(z\overline{a})}{z^m}e^{-\frac{1}{2}|z|^2}\right|^q d\mu(z) \lesssim(e^{\frac{p}{2}|a|^2})^\frac{q}{p}=e^{\frac{q}{2}|a|^2}.$$
In particular,
$$
\int_{B(a,r)}\left|\frac{e^{z\overline{a}}-p_m(z\overline{a})}{z^m}e^{-\frac{1}{2}|z|^2}\right|^q d\mu(z) \lesssim e^{\frac{q}{2}|a|^2}.$$
If $|a| > 2r$, then $|z|^m$ is comparable to $(1 + |a|)^m$ for $B(a,r)$. So
$$\int_{B(a,r)}|e^{z\overline{a}}|^q|1-e^{-z\overline{a}}p_m(z\overline{a})|^qe^{-\frac{q}{2}|z|^2} d\mu(z) \lesssim (1+|a|)^{mq}e^{\frac{q}{2}|a|^2}$$
holds for all $|a| > 2r$. Note that
$$
\lim_{|a|\to\infty}\inf_{z\in B(a,r)}|1-e^{-z\overline{a}}p_m(z\overline{a})| = 1.
$$
Thus
$$
\int_{B(a,r)}|e^{z\overline{a}}|^qe^{-\frac{q}{2}|z|^2}d\mu(z) \lesssim (1+|a|)^{mq}e^{\frac{q}{2}|a|^2}$$
holds for the sufficiently large $|a|$. But this last inequality is clearly true for smaller $|a|$ as
well. So we have
$$
\int_{B(a,r)}|e^{z\overline{a}}|^qe^{-\frac{q}{2}|z|^2} d\mu(z) \lesssim (1+|a|)^{mq}e^{\frac{q}{2}|a|^2}\quad\forall\quad a\in\mathbb C.
$$
Completing a square in the exponent, we can rewrite the inequality
above as $$\int_{B(a,r)}e^{-\frac{q}{2}|z-a|^2} d\mu(z) \lesssim (1+|a|)^{mq}$$
thereby deducing
$$\mu(B(a,r)) \lesssim(1+|a|)^{mq}e^{\frac{q}{2}r^2}\quad\forall\quad a\in\mathbb{C}.
$$

Conversely, assume that 
$$
\mu(B(a,r))\lesssim(1+|a|)^{mq}\quad\forall\quad a\in\mathbb{C}.
$$
We proceed to estimate the integral
$$
\|f\|^q_{L^q(\mathbb C,\mu)}=\int_{\mathbb{C}}|f(z)e^{-\frac{1}{2}|z|^2}|^q d\mu(z)
$$
of any given function $f\in \mathcal F^{p,m}$. For any positive number $s$ let $Q_s$ denote the following square in $\mathbb{C}$ with vertices $0, s, si$,
and $s + si$:
$$
Q_s = \{z = x + iy : 0 < x \le s\ \ \&\ \ 0 < y \le s\}.
$$
It is clear that
$$
\mathbb{C}=\cup_{a\in s\mathbb Z^2} (Q_s+a)
$$
is a decomposition of $\mathbb{C}$ into disjoint squares of side length $s$. Thus
$$
\|f\|^q_{L^q(\mathbb C,\mu)}=\sum_{a\in s\mathbb Z^2}\int_{Q_s+a}|f(z)e^{-\frac{1}{2}|z|^2}|^q d\mu(z).
$$

Fix positive numbers $s$ and $t$ such that $t+\sqrt{s} = r$. By Lemma \ref{l2a} (ii)
\begin{eqnarray*}
|f(z)e^{-\frac{1}{2}|z|^2}|^p &\lesssim& \int_{B(z,t)}|f(w)e^{-\frac{1}{2}|w|^2}|^p dA(w)\\ &\lesssim&\frac{1}{(1+|z|)^{mp}}\int_{B(z,t)}|w^mf(w)e^{-\frac{1}{2}|w|^2}|^p dA(w)
\end{eqnarray*}
holds for all $z\in\mathbb{C}$. Now if $z\in Q_s + a$, where
$a\in s\mathbb Z^2$ implies $B(z, t)\subset B(a, r)$ by the triangle inequality, and hence $1+|z|\approx 1+|a|$. Consequently, 
$$
|f(z)e^{-\frac{1}{2}|z|^2}|^p \lesssim\frac{1}{(1+|a|)^{mp}}\int_{B(a,r)}|w^mf(w)e^{-\frac{1}{2}|w|^2}|^p dA(w).
$$ 
This amounts to

$$
|f(z)e^{-\frac{1}{2}|z|^2}|^q \lesssim\frac{1}{(1+|a|)^{mq}}\left(\int_{B(a,r)}|w^mf(w)e^{-\frac{1}{2}|w|^2}|^p dA(w)\right)^{\frac{q}{p}}.
$$
Therefore,
$$
\|f\|^q_{L^q(\mathbb C,\mu)}\lesssim\sum_{a\in s\mathbb Z^2}\frac{\mu(B(a,r))}{(1+|a|)^{mq}}\left(\int_{B(a,r)}|w^mf(w)e^{-\frac{1}{2}|w|^2}|^p dA(w)\right)^{\frac{q}{p}}.
$$
Combining this last estimate with the previous assumption on $\mu$ and $p\le q$, we obtain
\begin{eqnarray*}
\|f\|^q_{L^q(\mathbb C,\mu)} &\lesssim& \sum_{a\in s\mathbb Z^2}\left(\int_{B(a,r)}|w^mf(w)e^{-\frac{1}{2}|w|^2}|^p dA(w)\right)^{\frac{q}{p}}\\
& \lesssim&\left(\sum_{a\in s\mathbb Z^2}\int_{B(a,r)}|w^mf(w)e^{-\frac{1}{2}|w|^2}|^p dA(w)\right)^{\frac{q}{p}}.
\end{eqnarray*}
Note that there exists a positive integer $N$ such that each point in $\mathbb C$ belongs to at most $N$ of the disks $B(a,r)$, where $a\in s\mathbb Z^2$. So, one gets
$$
\|f\|^q_{L^q(\mathbb C,\mu)}\lesssim\left(\int_{\mathbb{C}}|w^mf(w)e^{-\frac{1}{2}|w|^2}|^p dA(w)\right)^{\frac{q}{p}}\approx \|f\|_{\mathcal F^{p,m}}^q,
$$
as desired.

(ii) Suppose $0<q<p<\infty$. The following proof is inspired by \cite{X1}. 

First assume that $\|f\|_{L^q(\mathbb C,\mu)}\lesssim \|f\|_{\mathcal F^{p,m}}$ holds for all
$f\in \mathcal F^{p,m}$. For any $\{c_j\}\in l^p$, we may choose $\{r_j(t)\}$ as in Lemma \ref{l2b}, thereby getting 
$$
\{c_jr_j(t)\}\in l^p\quad\&\quad \|\{c_jr_j(t)\}\|_{l^p}=\|\{c_j\}\|_{l^p}.
$$ 
Then by Lemma \ref{l2a} (iii) we know that 
$$
\sum_{j=1}^\infty c_jr_j(t)k_{a_j}(z)\equiv z^mf(z)
$$ 
is in $\mathcal F^p$ with norm
$\|f\|_{\mathcal F^{p,m}}\approx\inf\|\{c_j\}\|_{l^p}$. Here $\{a_j\}$ is the sequence of all complex numbers of $s\mathbb Z^2$ and $k_a(z)=e^{z\overline{a}-\frac{1}{2}|a|^2}$.
In particular, 
$$
f(z)=\sum_{j=1}^\infty c_jr_j(t){k_{a_j}(z)}z^{-m}.
$$
According to the assumption we have
$$
\int_{\mathbb C}\Big|\sum_{j=1}^\infty c_jr_j(t)\frac{k_{a_j}(z)}{z^m e^{|z|^2/2}}\Big|^qd\mu(z)=
\|f\|^q_{L^q(\mathbb C,\mu)}\lesssim\|f\|_{\mathcal F^{p,m}}^q,
$$
whence getting by Lemma \ref{l2b},
$$
\int_\mathbb{C}\Big(\sum_{j=1}^\infty {|c_jk_{a_j}(z)e^{-\frac{1}{2}|z|^2}|^2}{|z|^{-2m}}\Big)^\frac{q}{2}d\mu(z)\lesssim\|f\|_{\mathcal F^{p,m}}^q.
$$
Also, note that if $|a|> 2r$ then $|z|^m$ is comparable to $(1+|a|)^m$ for $z\in B(a,r)$. So
\begin{eqnarray*}
&&\int_\mathbb{C}\Big(\sum_{j=1}^\infty {|c_jk_{a_j}(z)e^{-\frac{1}{2}|z|^2}|^2}{|z|^{-2m}}\Big)^\frac{q}{2}d\mu(z)\\
&&=\sum_{l=1}^\infty\int_{Q_s+a_l}\Big(\sum_{j=1}^\infty {|c_je^{-\frac{1}{2}|z-a_j|^2}|^2}{|z|^{-2m}}\Big)^\frac{q}{2}d\mu(z)\\
&&\ge\sum_{l=1}^\infty\int_{Q_s+a_l}{|c_l|^q}{|z|^{-mq}}e^{-\frac{q}{2}|z-a_l|^2} d\mu(z)\\
&&\gtrsim\sum_{j=1}^\infty\int_{B(a_j,r)}{|c_j|^q}{|z|^{-mq}}e^{-\frac{q}{2}|z-a_j|^2} d\mu(z)\\
&&\gtrsim\sum_{j=1}^\infty |c_j|^q\frac{\mu(B(a_j,r))}{(1+|a_j|)^{mq}}.
\end{eqnarray*}
So, a combination of the previously-established inequalities gives
$$
\sum_{j=1}^\infty |c_j|^q\frac{\mu(B(a_j,r))}{(1+|a_j|)^{mq}}\lesssim\|\{|c_j|\}\|^q_{l^p}=\|\{|c_j|^q\}\|_{l^{p/q}}.
$$
Since $p/(p-q)$ is the conjugate number of $p/q$, an application of the Riesz representation theorem yields 
$$
\left\{\frac{\mu(B(a_j,r))}{(1+|a_j|)^{mq}}\right\}\in l^{\frac{p}{p-q}}.
$$

Conversely, assume that the last statement holds. Note that the first part of the argument for the above (i) tells that
$$
\|f\|^q_{L^q(\mathbb C,\mu)}\lesssim\sum_{a\in s\mathbb Z^2}\frac{\mu(B(a,r))}{(1+|a|)^{mq}}\left(\int_{B(a,r)}|w^mf(w)e^{-\frac{1}{2}|w|^2}|^p dA(w)\right)^{\frac{q}{p}}
$$
holds for all $f\in \mathcal F^{p,m}$. Applying H\"{o}lder's inequality to the last summation we obtain
\begin{eqnarray*}
\|f\|^q_{L^q(\mathbb C,\mu)} & \lesssim& \left(\sum_{a\in s\mathbb Z^2}\left(\frac{\mu(B(a,r))}{(1+|a|)^{mq}}\right)^{\frac{p}{p-q}}\right)^{\frac{p-q}{p}}\\
&&\quad\times \left(\sum_{a\in s\mathbb Z^2}\int_{B(a,r)}|w^mf(w)e^{-\frac{1}{2}|w|^2}|^p dA(w)\right)^{\frac{q}{p}}.
\end{eqnarray*}
Once again, notice that there exists a positive integer $N$ such that each point in $\mathbb C$ belongs to at most $N$ of the disks $B(a,r)$, where $a\in s\mathbb Z^2$. So, 
$$
\|f\|^q_{L^q(\mathbb C,\mu)}\lesssim\left(\sum_{a\in s\mathbb Z^2}\left(\frac{\mu(B(a,r))}{(1+|a|)^{mq}}\right)^{\frac{p}{p-q}}\right)^{\frac{p-q}{p}}\|f\|_{\mathcal F^{p,m}}^q.
$$
This completes the argument.
\end{proof}

The following extends \cite[Theorem 5]{CMS} (cf. \cite[Theorem 1]{U}), and \cite[Theorem 1]{Co}, respectively.

\begin{cor}\label{c2a} Let $\phi:\mathbb C\mapsto\mathbb C$ be an entire function. For $p\in (0,\infty)$ and a nonnegative integer $m$ define two linear operators acting on an entire function $f:\mathbb C\mapsto\mathbb C$:
$$
\begin{cases}
C_\phi f(z)=f\circ\phi(z)\quad\forall\quad z\in\mathbb C;\\
T_\phi f(z)=\int_0^z f(w)\phi'(w)\,dw\quad\forall\quad z\in\mathbb C.
\end{cases}
$$

\item{\rm(i)} The composition operator $C_\phi: \mathcal F^{p,m}\mapsto \mathcal F^{q}$ exists as a bounded operator if and only if
$$
\begin{cases}
\sup_{a\in\mathbb C}\frac{\int_{\phi^{-1}(B(a,r))}e^{-q|z|^2/2}\,dA(z)}{(1+|a|)^{mq}}<\infty\ \ \hbox{when}\ \ 0<p\le q<\infty;\\
\int_{\mathbb C}\Big(\frac{\int_{\phi^{-1}(B(a,r))}e^{-q|z|^2/2}\,dA(z)}{(1+|a|)^{mq}}\Big)^{p/(p-q)}\,dA(a)<\infty\ \ {when}\ \ 0<q<p<\infty.
\end{cases}
$$

\item{\rm(ii)} The Riemann-Stieltjes integral operator $T_\phi: \mathcal F^{p,m}\mapsto \mathcal F^{q}$ exists as a bounded operator if and only if
$$
\begin{cases}
\sup_{a\in\mathbb C}\frac{\int_{B(a,r)}\big(\frac{|\phi'(z)|}{(1+|z|)}\big)^q\,dA(z)}{(1+|a|)^{mq}}<\infty\ \ \hbox{when}\ \ 0<p\le q<\infty;\\
\int_{\mathbb C}\Big(\frac{\int_{B(a,r)}\big(\frac{|\phi'(z)|}{(1+|z|)}\big)^q\,dA(z)}{(1+|a|)^{mq}}\Big)^{p/(p-q)}\,dA(a)<\infty\ \ \hbox{when}\ \ 0<q<p<\infty.
\end{cases}
$$
\end{cor} 

\begin{proof}
(i) For any Borel set $E\subset\mathbb C$ let $\phi^{-1}(E)$ be the pre-image of $E$ under $\phi$ and 
$$
\mu(E)=\int_{\phi^{-1}(E)}\exp\Big(-\frac{q|z|^2}{2}\Big)\,dA(z).
$$
Then 
$$
\|C_\phi f\|^q_{L^q(\mathbb C,\mu)}=\int_{\mathbb C}|f(z)e^{-\frac{|z|^2}{2}}|^q\,d\mu(z)\quad\forall\quad f\in \mathcal F^{p,m}.
$$
An application of Theorem \ref{t2a} with the above formula gives the desired result.

(ii) According to \cite[Proposition 1]{Co}, an entire function $f:\mathbb C\mapsto\mathbb C$ belongs to $\mathcal F^{q}$ if and only if
$$
\int_{\mathbb C}\Big(\frac{|f'(z)|e^{-\frac{|z|^2}{2}}}{1+|z|}\Big)^q\,dA(z)<\infty.
$$
So, $T_\phi f\in \mathcal F^q$ is equivalent to
$$
\int_{\mathbb C}\Big(\frac{|f(z)\phi'(z)|e^{-\frac{|z|^2}{2}}}{1+|z|}\Big)^q\,dA(z)<\infty.
$$
Now, choosing
$$
d\mu(z)=\Big(\frac{|\phi'(z)|}{1+|z|}\Big)^q\,dA(z)
$$
in Theorem \ref{t2a}, we get the boundedness result for $T_\phi$. 
\end{proof}

\section{Convexities or concavities in logarithm}\label{s3}

We also need two lemmas. The first one comes directly from \cite[Lemmas 2, 1, 6]{WZ} with $(0,1)$ being replaced by $(0,\infty)$. 
\begin{lem}\label{l3a}
\item{\rm(i)} Suppose $f$ is positive and twice differentiable on $(0,\infty)$. Let
$$
D(f(x))\equiv\frac{f'(x)}{f(x)}
+x\frac{f''(x)}{f(x)}-x\left(\frac{f'(x)}{f(x)}\right)^2.$$ Then the function $\ln f(x)$ is
concave in $\ln x$ if and only if $D(f(x))\le0$
on $(0,\infty)$ and $\ln f(x)$ is
convex in $\ln x$ if and only if $D(f(x))\ge0$
on $(0,\infty)$.

\item{\rm(ii)} Suppose $f$ is twice differentiable on $(0,\infty)$. Then $f(x)$ is
convex in $\ln x$ if and only if $f(x^2)$ is convex in $\ln x$ and $\ln f(x)$ is
concave in $\ln x$ if and only if $f(x^2)$ is concave in $\ln x$.

\item{\rm(iii)} Suppose $\{h_k(x)\}$ is a sequence of positive and twice
differentiable functions on $(0,\infty)$ such that the function
$$
H(x)=\sum_{k=0}^\infty h_k(x)
$$
is also twice differentiable on $(0,\infty)$. If for each natural number $k$ the
function $\ln h_k(x)$ is convex in $\ln x$, then $\ln H(x)$ is
also convex in $\ln x$. 
\end{lem}

The second lemma as below is elementary.

\begin{lem}\label{l3b} Suppose $f$ is continuous differentiable on $[0,\infty)$. If $f'(\infty)\equiv\lim_{x\to\infty}f'(x)=-\infty$, then $f(\infty)\equiv\lim_{x\to\infty}f(x)=-\infty$.
\end{lem}

The main result of this section is the following log-convexity theorem.

\begin{thm}\label{t3a} Suppose $k$ is a nonnegative integer and $0<p<\infty$. 

\item{\rm(i)} If $0<\alpha<\infty$, then the function $r\mapsto\ln\mathsf{M}_{p,\alpha}(z^k,r)$ is concave in $\ln r$.

\item{\rm(ii)} If $-\infty<\alpha\le 0$, then there exists some $c$ (depending on $k$ and $\alpha$)
on $(0,\infty)$ such that the function $r\mapsto \ln\mathsf{M}_{p,\alpha}(z^k,r)$ is convex in $\ln r$ on $(0,c]$ and concave in $\ln r$ on $[c,\infty)$.
\label{4}
\end{thm}

\begin{proof} The case $\alpha=0$ is a straightforward by-product of the classical Hardy convexity theorem (cf. \cite{Ta}). So, for the rest of the proof we may assume $\alpha\neq0$.

By the polar coordinates and an obvious change of variables, we have
$$
\mathsf{M}_{p,\alpha}(z^k,r)=\frac{\int_0^{r^2}t^{pk/2}e^{-\alpha t}\,dt}
{\int_0^{r^2}e^{-\alpha t}\,dt}.
$$

For any nonnegative parameter $\lambda$ we define
\begin{equation*}
f_\lambda(x)=\int_0^xt^\lambda e^{-\alpha t}\,dt\quad\forall x\in (0,\infty).
\end{equation*}
To prove Theorem \ref{t3a}, by Lemma \ref{l3a} (i)-(ii), we need only to consider the function
\begin{equation*}
\Delta(\lambda,x)=\frac{f_\lambda'}{f_\lambda}+x\frac{f_\lambda''}{f_\lambda}
-x\left(\frac{f_\lambda'}{f_\lambda}\right)^2
-\left(\frac{f'_0}{f_0}+x\frac{f_0''}{f_0}-x\left(\frac{f_0'}{f_0}\right)^2\right).
\end{equation*}
Here and henceforth, the derivatives $f'_\lambda(x)$ and
$f''_\lambda(x)$ are taken with respect to $x$ not $\lambda$.

To simplify notation, we write $h=f_\lambda(x)$ and denote by $h'$, $h''$, $h'''$ to
the various derivatives of $f_\lambda(x)$ with respect to $x$. Meanwhile, $\partial/\partial\lambda$ stands for the derivative with respect to $\lambda$.

Thanks to
$h=\int_0^xt^\lambda e^{-\alpha t}\,dt$, we get
$$
\begin{cases}
h'=x^\lambda e^{-\alpha x};\\
h''=(\lambda-\alpha x)x^{\lambda-1}e^{-\alpha x};\\
h'''=x^{\lambda-2}e^{-\alpha x}\left(\lambda^2-\lambda-2\lambda\alpha x+\alpha^2
x^2\right).
\end{cases}
$$
At the same time, we have
$$
\begin{cases}
\frac{\partial h}{\partial\lambda}=\int_0^xt^\lambda e^{-\alpha t}\,\ln t\,dt;\\
\frac{\partial h'}{\partial\lambda}=\frac{\partial}{\partial x}\left(\frac{\partial h}{\partial \lambda}\right)=h'\ln x;\\
\frac{\partial h''}{\partial\lambda}=\frac{h'}{x}+h''\ln x.
\end{cases}
$$
Note that the function inside the brackets in $\Delta(\lambda,x)$ is independent of $\lambda$. So,
\begin{eqnarray*}
&&\frac{\partial\Delta}{\partial\lambda}\\
&&=\frac{1}{h^2}\left(h\frac{\partial h'}{\partial\lambda}+xh\frac{\partial h''}{\partial \lambda}-2xh'\frac{\partial h'}{\partial\lambda}\right)\\
&&\quad-\frac{1}{h^3}\frac{\partial h}{\partial \lambda}\left(hh'+xhh''-2x(h')^2\right)\\
&&=\frac{1}{h^2}\left(hh'\ln x+hh'+xhh''\ln x-2x(h')^2\ln x\right)\\
&&\quad-\frac{1}{h^3}\frac{\partial h}{\partial \lambda}\left(hh'+xhh''-2x(h')^2\right)\\
&&=\frac{h'}{h}+\frac{1}{h^3}\left(h\ln x-\frac{\partial h}{\partial\lambda}\right)(hh'+xhh''-2x(h')^2).
\end{eqnarray*}

From now on, we use the notation $\mathsf{X}\sim\mathsf{Y}$ to represent that $\mathsf{X}$ and $\mathsf{Y}$ have the same sign. Let us consider the following two functions (with $\lambda$ fixed):
$$
\begin{cases}
d_1(x)=h\ln x-\frac{\partial h}{\partial\lambda};\\
d_2(x)=\frac{hh'+xhh''-2x(h')^2}{h'}=(\lambda+1-\alpha x)h-2x^{\lambda+1}e^{-\alpha x}.
\end{cases}
$$
Since $d_1'(x)=h/x>0$, one has $d_1(x)\ge d_1(0)=0$. Now we want to prove that $d_2(x)<0$ for all $x>0$. By direct computations, we obtain
$$
\begin{cases}
d_2'(x)=-\alpha h-(\lambda+1-\alpha x)x^\lambda e^{-\alpha x};\\
d_2''(x)=(\lambda+1-\alpha x)(-\lambda+\alpha x)x^{\lambda-1} e^{-\alpha x}.
\end{cases}
$$

(i) If $\alpha>0$, then under $0<x\le\frac{\lambda+1}{\alpha}$ we have $d_2'(x)\le -\alpha h<0$. When $x>\frac{\lambda+1}{\alpha}$, it is easy to obtain
$d_2''(x)<0$, and then
$$
d_2'(x)\le d_2'\left(\frac{\lambda+1}{\alpha}\right)=-\alpha h\left(\frac{\lambda+1}{\alpha}\right)<0.
$$ 
Hence $d_2(x)<d_2(0)=0$ for all $x>0$.

(ii) If $\alpha<0$, then it is easy to see 
$d_2''(x)<0$, and hence
$d_2'(x)\le d_2'(0)=0$. This in turn implies $d_2(x)<d_2(0)=0$ for all $x>0$.

With the help of the above analysis, we deduce 
$$
\frac{\partial\Delta}{\partial\lambda}\sim -\frac{h^2h'}{hh'+xhh''-2x(h')^2}-h\ln x+\frac{\partial h}{\partial\lambda}=:\delta(x).
$$
Further computations derive
\begin{eqnarray*}\label{eq30}
&&\delta'(x)\\
&&=-\frac{2h(h')^2+h^2h''}{hh'+xhh''-2x(h')^2}\nonumber\\
&&\ +\frac{h^2h'(2hh''+xhh'''-3xh'h''-(h')^2)}{(hh'+xhh''-2x(h')^2)^2}-\frac{h}{x}\nonumber\\
&&=\left(\frac{h^2}{x(hh'+xhh''-2x(h')^2)^2}\right)\nonumber\\
&&\ \times\left(-\left((h')^2+xh'h''+2x^2(h'')^2-x^2h'h'''\right)h+x(h')^2(h'+xh'')\right)\nonumber\\
&&=\left(\frac{hh'}{hh'+xhh''-2x(h')^2}\right)^2\\
&&\ \times\left(\frac{\left((\lambda+1)^2-(2\lambda+1)\alpha x+
\alpha^2x^2\right)\delta_1(x)}{x}\right).\nonumber\\
\end{eqnarray*}
Here
$$
\delta_1(x)=-h+\frac{x^{\lambda+1}e^{-\alpha x}(\lambda+1-\alpha x)}{(\lambda+1)^2-(2\lambda+1)\alpha x+
\alpha^2x^2}.
$$ 
And, a computation implies 
\begin{eqnarray*}
\delta_1'(x)&=&\frac{-\alpha x^{\lambda+1}e^{-\alpha x}(\lambda+1+\alpha x)}{\left((\lambda+1)^2-(2\lambda+1)\alpha x+
\alpha^2x^2\right)^2}
\end{eqnarray*}
and then
$$
\delta_1'(0)=\delta_1(0)=0.
$$
With details deferred to after the proof, we also have
$\delta'(0)=0$ and when $\alpha<0$ we have $\delta'(\infty)=-\infty$. Without loss of generality, we may just handle the case $\lambda>0$ in what follows.

(i) If $\alpha>0$, then $\delta_1'(x)<0$ for
all $x\in(0,\infty)$, and hence $\delta_1(x)<\delta_1(0)=0$ on $(0,\infty)$. This implies 
$$
\frac{\partial\Delta(\lambda,x)}{\partial\lambda}<0\qquad\forall\quad x\in(0,\infty).
$$
Therefore, $\Delta(\lambda,x)\le\Delta(0,x)=0$, and the desired result follows.

(ii) If $\alpha<0$, then $\delta_1'(x)$ has only one zero $-(\lambda+1)/\alpha$
on $(0,\infty)$ and $\delta_1(x)$ is increasing on $(0,-(\lambda+1)/\alpha)$ and decreasing on $(-(\lambda+1)/\alpha,\infty)$. Noticing $\delta_1'(\infty)=-\infty,$ we use Lemma \ref{l3b} to get $\delta_1(\infty)=-\infty.$ Hence
$\delta_1(x)$ has only one zero $x^*$
on $(0,\infty)$ (Note that $x^*>\frac{\lambda+1}{-\alpha}$.) and $\delta_1(x)$ is positive on $(0,x^*)$ and negative on $(x^*,\infty)$. For $\delta(x)$ and $\frac{\partial\Delta}{\partial\lambda}$ we have similar results. Hence $\frac{\partial\Delta}{\partial\lambda}$ has exactly one zero $x_0$ (depending on $k$ and $\alpha$)
on $(0,\infty)$ (Note that $x_0>x^*>\frac{\lambda+1}{-\alpha}$.) and $\frac{\partial\Delta}{\partial\lambda}$ is positive on $(0,x_0)$ and negative on $(x_0,\infty)$. This implies  $\Delta(\lambda,x)\ge\Delta(0,x)=0$ on $(0,x_0)$ and $\Delta(\lambda,x)\le\Delta(0,x)=0$ on $(x_0,\infty)$. Now, setting $c=\sqrt{x_0}$ yields the desired result.

Finally, let us verify the above-claimed formulas:
$$
\begin{cases}
\delta'(0)=0;\\
\delta'(\infty)=-\infty.
\end{cases}
$$ 
As a matter of fact, L'Hopital's rule gives
$$
\lim_{x\to0}\frac hx=0,\quad\lim_{x\to0}\frac{xh'}{h}=\lim_{x\to0}\frac{h'+xh''}{h'}=\lambda+1.
$$
Consequently,
$$
\lim_{x\to0}\frac{hh'}{hh'+xhh''-2x(h')^2}=\lim_{x\to0}\frac{1}{\frac{h'+xh''}{h'}-2\frac{xh'}{h}}=
-(\lambda+1)^{-1}.
$$ 
It follows from the definition of $\delta_1(x)$ that
${\lim_{x\to0}\frac{\delta_1(x)}{x}}=0.$ So, by the definition of $\delta_1(x)$ we have $\delta'(0)=0$.

In a similar manner, another application of L'Hopital's rule derives
$$
\lim_{x\to\infty}\frac{h'}{h}=\lim_{x\to\infty}\frac{h''}{h'}=-\alpha
$$
and consequently,
$$
\lim_{x\to\infty}\frac{xhh'}{hh'+xhh''-2x(h')^2}=\lim_{x\to\infty}\frac{1}{\frac{1}{x}+\frac{h''}{h'}-2\frac{h'}{h}}=
\frac1{\alpha}.
$$ 
The definition of $\delta_1(x)$ and L'Hopital's rule
imply 
$$
\lim_{x\to\infty}\frac{\delta_1(x)}{x}=\lim_{x\to\infty}\delta_1'(x)=-\infty,
$$
and then $\delta'(\infty)=-\infty$.
\end{proof}

\begin{cor}\label{c3a} Suppose $\alpha<0$ and $0<p<\infty$. Then the function $r\mapsto \ln\mathsf{M}_{p,\alpha}(z^k,r)$ is convex in $\ln r$ on $(0,c]$, where $c=\sqrt{{(pk+2)}{(-2\alpha)^{-1}}}$. Moreover, for $p=2$, the function $r\mapsto \ln \mathsf{M}_{2,\alpha}(f,r)$ is convex in $\ln r$ on $(0,\sqrt{{(-\alpha)^{-1}}}]$ for any entire function $f:\mathbb C\mapsto\mathbb C$.
\end{cor}
\begin{proof} The first part of Corollary \ref{c3a} follows from the proof of Theorem \ref{t3a} with $\lambda=pk/2$. As for the second part, it is easy to see $c\ge\sqrt{\frac{1}{-\alpha}}$. Suppose 
$$
f(z)=\sum_{k=0}^\infty a_kz^k.
$$
It follows from an integration in polar coordinates that
$$
\mathsf{M}_{2,\alpha}(f,r)=\sum_{k=0}^\infty|a_k|^2\mathsf{M}_{2,\alpha}(z^k,r).
$$
Now, applying Lemma \ref{l3a} (iii) we obtain the desired result.
\end{proof}

\begin{rem}\label{r3a} Theorem \ref{t3a} tells us that the integral means of all monomials are logarithmically concave when $\alpha>0$. However, this is not true for all entire functions, even for linear mappings.
\end{rem}
\begin{proof} For instance, just choose $p=2,\alpha=1$ and $f(z)=a+z$. Using polar coordinates and changing variables we have
$$
\mathsf{M}_{p,\alpha}(f,r)=\frac{\int_0^{r^2}(c+t)e^{-t}\,dt}
{\int_0^{r^2} e^{-t}\,dt}\quad\hbox{where}\quad c=|a|^2.
$$
By Lemma \ref{l3a}(i), we just need to consider the function
$$
F(x)=\frac{\int_0^x(c+t)e^{-t}\,dt}
{\int_0^x e^{-t}\,dt}=\frac{c+1-(c+1+x)e^{-x}}{1-e^{-x}}\equiv\frac{g(x)}{h(x)}.
$$
Employing the $D$-notation in Lemma \ref{l3a} (i), we have
$$
\begin{cases}
D(g(x))=\frac{1}{g^2}\left((c+2x-cx-x^2)e^{-x}g(x)-x(c+x)^2e^{-2x}\right)\\
\hbox{and}\\
D(h(x))=\frac{e^{-x}}{h^2}(1-x-e^{-x}),
\end{cases}
$$
whence getting
\begin{eqnarray*}
&&D(F(x))\\
&&=D(g(x))-D(h(x))\\
&&=e^{-x} \left((c+1)(-1+3x-x^2)+(3+3c-6x-6cx-x^2)e^{-x}\right.\\
&&\ \left. +(-3-3c+3x+3cx+2x^2+cx^2+x^3)e^{-2x}+(c+1)e^{-3x}\right)\\
&&\sim (c+1)(-1+3x-x^2)e^{3x}+(3+3c-6x-6cx-x^2)e^{2x}\\
&&\ +(-3-3c+3x+3cx+2x^2+cx^2+x^3)e^{x}+(c+1)\\
&&\equiv G(x).
\end{eqnarray*}
A direct computation gives
$$
\begin{cases}
G'(x)\\
\sim(c+1)(7-3x)-\frac{14+12c+2x}{e^x}+\frac{7+5c+5x+cx+x^2}{e^{2x}}\equiv H(x);\\
H'(x)=-3(c+1)+\frac{12+12c+2x}{e^x}-\frac{9+9c+8x+2cx+2x^2}{e^{2x}};\\
H''(x)\sim-10-12c-2x+\frac{10+16c+12x+4cx+4x^2}{e^x}\equiv J(x);\\
J'(x)=-2(1-e^{-x})-(12c+4x+4cx+4x^2)e^{-x}\le 0.
\end{cases}
$$
Noticing that $J(0)=4c>0$ and $J(\infty)=-\infty$, we know that there exists a number $x_3\in (0,\infty)$ such that $J(x)$ is positive, and hence $H''(x)$ is positive on $(0,x_3)$ and negative on $(x_3,\infty)$. Note that 
$$
\begin{cases}
H'(0)=H(0)=G(0)=0;\\
H'(\infty)=-3(c+1)<0;\\
H(\infty)=G(\infty)=-\infty.
\end{cases}
$$ 
So, the functions $H'',~ H', H, G$ have similar properties. In particular, there exists a $x_0\in (0,\infty)$ such that $G(x)$, and hence, $D(F(x))$ is positive on $(0,x_0)$ and negative on $(x_0,\infty)$. This implies that $\ln\mathsf{M}_{p,\alpha}(f,r)$ is convex in $\ln r$ on $(0,\sqrt{x_0})$ and concave in $\ln r$ on $(\sqrt{x_0},\infty)$. Especially, when $c=0$, the function $\ln\mathsf{M}_{p,\alpha}(f,r)$ is concave in $\ln r$ on $(0,\infty)$.

Certainly, it is interesting to determine the maximal interval $(\lambda,\infty)$ on which $G$ is negative, that is, the area integral means $\mathsf{M}_{p,\alpha}(a+z,r)$ is logarithmically concave on $(\sqrt{\lambda},\infty)$ for any $a\in\mathbb{C}$.

It follows from the definition of $G$ that
$$
G(x)\sim G_0(x)+\frac{x^2e^x}{c+1}(1+x-e^x),
$$
where 
$$
G_0(x)=(-1+3x-x^2)e^{3x}+(3-6x)e^{2x}+(-3+3x+x^2)e^{x}+1.
$$
Note that 
$$
G(x)<0\ \ \forall\ \ c=|a|^2\Longleftrightarrow G_0(x)<0.
$$ 
Also, it is not hard to prove that $G_0(x)$ has exactly one real zero $\lambda$ in $(0,\infty)$, and $G_0(x)$ is positive on $(0,\lambda)$ and negative on $(\lambda,\infty)$. A numerical computation shows that $\lambda=1.86047095\cdots$. This implies that $\mathsf{M}_{2,1}(a+z,r)$ is logarithmically concave on $(\sqrt{\lambda},\infty)$ for any $a\in\mathbb{C}$, and the interval $(\sqrt{\lambda},\infty)$ is maximal.
\end{proof}

\end{document}